\documentclass{amsart}
\usepackage{amsmath, amssymb}
\usepackage[colorlinks=true]{hyperref}

\theoremstyle{plain}
 \newtheorem{theorem}{Theorem}[section]
 \newtheorem{lemma}{Lemma}[section]
 \newtheorem{proposition}{Proposition}[section]
 \newtheorem{corollary}{Corollary}[section]

\theoremstyle{definition}

\newcommand{\Z}{\mathbb{Z}}
\newcommand{\R}{\mathbb{R}}
\newcommand{\transpose}{\mathsf{T}}
\newcommand{\smallmat}[4]{\left(\begin{smallmatrix} #1 & #2 \\ #3 & #4 \end{smallmatrix}\right)}
\newcommand{\smallvec}[2]{\left(\begin{smallmatrix} #1 \\ #2 \end{smallmatrix}\right)}

\title[Largest values of the Stern sequence, alternating binary \ldots]{Largest values of the Stern sequence, alternating binary expansions and continuants}

\author{Roland Paulin}
\address{Roland Paulin, Department of Mathematics, University of Salzburg, Hellbrunnerstr.\ 34/I, 5020 Salzburg, Austria}
\email{paulinroland@gmail.com}
\thanks{The author was supported by the Austrian Science Fund (FWF): P24574.}

\date{\today}

\begin{document}

\begin{abstract}
We study the largest values of the $r$th row of Stern's diatomic array.
In particular, we prove some conjectures of Lansing.
Our main tool is the connection between the Stern sequence, alternating binary expansions and continuants.
This allows us to reduce the problem of ordering the elements of the Stern sequence to the problem of ordering continuants.
We describe an operation that increases the value of a continuant, allowing us to reduce the problem of largest continuants to ordering continuants of very special shape.
Finally, we order these special continuants using some identities and inequalities involving Fibonacci numbers.
\end{abstract}

\maketitle

%==========================================
\section{Introduction}
%==========================================

The Stern sequence $(s(n))_{n \ge 0}$ is defined as follows: $s(0) = 0$, $s(1) = 1$, $s(2n) = s(n)$ and $s(2n+1) = s(n) + s(n+1)$ for every $n \ge 0$.
Stern's diatomic array consists of rows indexed by $0,1,2, \dotsc$, where if $r \ge 0$, then the $r$th row is $s(2^r), s(2^r+1), \dotsc, s(2^{r+1})$.
Stern's diatomic array can also be constructed in the following way.
Start with the $0$th row $1,1$.
If $r \ge 1$, then to construct the $r$th row, copy the previous row, and between every two consecutive numbers $x,y$, write their sum $x+y$.
This array was first studied by Stern in \cite{Stern}.
Lehmer summarized several properties of this array in \cite{Lehmer}.

For $r \ge 0$ and $m \ge 1$, let $L_m(r)$ denote the $m$th largest distinct value of the $r$th row of Stern's diatomic array.
If there are less than $m$ distinct values in the $r$th row, then we just take $L_m(r) = -\infty$.
Lucas in \cite{Lucas} has observed that the largest value in the $r$th row is $L_1(r) = F_{r+2}$, where $F_n$ denotes the $n$th Fibonacci number.
Lansing in \cite{Lan14} determined the second and third largest values $L_2(r)$ and $L_3(r)$, and formulated the following conjectures about the $L_m(r)$'s.
Conjecture 7 of \cite{Lan14} says that if $m \ge 1$ and $r \ge 4m-2$, then
\[
L_m(r) = L_m(r-1) + L_m(r-2).
\]
Conjecture 9 of \cite{Lan14} says that if $m \ge 2$ and $r \ge 4m-4$, then
\[
L_m(r) = L_{m-1}(r) - F_{r-(4m-5)} = F_{r+2} - \sum_{j=2}^m F_{r-(4j-5)}.
\]

The following theorem is our main result.
It gives a formula for $L_m(r)$ for certain values of $r$ and $m$, and it implies the two conjectures of Lansing stated above.
\begin{theorem} \label{thm:main}
If $r \in \Z_{\ge 0}$, then
\[
\{L_1(r), \dotsc, L_{\lceil \frac{r}{2} \rceil}(r)\} = \{F_{r+2} - F_i F_j ; \, i,j \in \Z_{\ge 0}, \, i+j = r-1 \}.
\]
More explicitly, if $1 \le m \le \lceil\frac{r}{2}\rceil$, then
\[
L_m(r) = F_{r+2} - F_{2m-2-b} F_{r-2m+1+b},
\]
where
\[
b = b(m,r) = \begin{cases}
0 & \textrm{ if } m \le \lfloor \frac{r+3}{4} \rfloor \textrm{ or } 2 \mid r, \\
1 & \textrm{ if } m > \lfloor \frac{r+3}{4} \rfloor \textrm{ and } 2 \nmid r.
\end{cases}
\]
\end{theorem}

We show now that this theorem implies the two conjectures stated above.
\begin{proof}[\protect{Proof of \cite[Conjecture 7]{Lan14}}]
Let $m \ge 1$ and $r \ge 4m-2$.
Then $m \le \lfloor \frac{(r-1)+3}{4} \rfloor$, so $L_m(r) = F_{r+2} - F_{2m-2} F_{r-2m+1}$ and $L_m(r-1) = F_{r+1} - F_{2m-2} F_{r-2m}$ by Theorem \ref{thm:main}.
If $r \ge 4m-1$, then $m \le \lfloor \frac{(r-2)+3}{4} \rfloor$, while if $r = 4m-2$, then $2 \mid r-2$, so either way $L_m(r-2) = F_r - F_{2m-2} F_{r-2m-1}$.
The conjecture now follows from
\[
F_{r+2} - F_{2m-2} F_{r-2m+1} = (F_{r+1} - F_{2m-2} F_{r-2m}) + (F_r - F_{2m-2} F_{r-2m-1}).
\]
\end{proof}

\begin{proof}[\protect{Proof of \cite[Conjecture 9]{Lan14}}]
It is enough to prove $L_{m-1}(r) - L_m(r) = F_{r-(4m-5)}$, because $L_1(r) = F_{r+2}$. 
Now $m-1 \le \lfloor \frac{r+3}{4} \rfloor$, and either $2 \mid r = 4m-4$, or $m \le \lfloor \frac{r+3}{4} \rfloor$.
So $L_{m-1}(r) = F_{r+2} - F_{2m-4} F_{r-2m+3}$ and $L_m(r) = F_{r+2} - F_{2m-2} F_{r-2m+1}$ by Theorem \ref{thm:main}.
Hence
\begin{align*}
L_{m-1}(r) - L_m(r) &= F_{2m-2} F_{r-2m+1} - F_{2m-4} F_{r-2m+3} = (-1)^{2m-4} F_2 F_{r-(4m-5)} \\
&= F_{r-(4m-5)}
\end{align*}
by Lemma \ref{lemma:FiFj-FkFl}.
\end{proof}

Here is a brief description of the contents of the paper.
In section \ref{sec:prelim} we state some basic results about Fibonacci numbers, alternating binary expansions and continuants.
In section \ref{sec:connect} we describe the connection between the Stern sequence, alternating binary expansions and continuants.
Using this description we reduce the problem of comparing the elements of the Stern sequence to the problem of comparing continuants, which is the subject of section \ref{sec:comparing}.
Using some identities involving Fibonacci numbers, we finish the proof of Theorem \ref{thm:main} in section \ref{sec:Fibonacci}.
Finally, in section \ref{sec:further-research} we discuss possible extensions of our results.

%==========================================
\section{Preliminaries} \label{sec:prelim}
%==========================================

In this section we introduce some notations, and state some basic results about Fibonacci numbers, alternating binary expansions and continuants.
Let $F_n$ denote the $n$th Fibonacci number for every $n \in \Z$.
So $F_0 = 0$, $F_1 = 1$, and $F_n = F_{n-1} + F_{n-2}$ for every $n \in \Z$.
Note that $\smallmat{1}{1}{1}{0}^n = \smallmat{F_{n+1}}{F_n}{F_n}{F_{n-1}}$ for every $n \in \Z$.
The following basic lemma describes the sign of the Fibonacci numbers.
\begin{lemma} \label{lemma:Fibonacci-sign}
If $n \in \Z$, then $F_{-n} = (-1)^{n+1} F_n$.
Hence 
\[
F_n \begin{cases}
=0 & \textrm{ if } n = 0, \\
>0 & \textrm{ if } n>0 \textrm{ or } 2 \nmid n, \\
<0 & \textrm{ if } n<0 \textrm{ and } 2 \mid n.
\end{cases}
\]
\end{lemma}
\begin{proof}
The first part is easy to check by induction on $n$, and the second part follows from the first part.
\end{proof}

The following lemma describes a useful identity, which allows us to compare $F_i F_j$'s with fixed $i+j$.
\begin{lemma} \label{lemma:FiFj-FkFl}
If $i,j,k,l \in \Z$ and $i+j = k+l$, then
\[
F_i F_j - F_k F_l = (-1)^k F_{i-k} F_{j-k}.
\]
Hence if $i+j = k+l$, then $F_i F_j = F_k F_l$ if and only if $\{i,j\} = \{k,l\}$.
\end{lemma}
\begin{proof}
Vajda's identity (see \cite{wikiVajda}) says that
\[
F_{n+i} F_{n+j} - F_n F_{n+i+j} = (-1)^n F_i F_j
\]
for every $n,i,j \in \Z$.
Substituting $k$, $i-k$ and $j-k$ into $n$, $i$ and $j$, we obtain the stated identity.
To prove the second part, note that by Lemma \ref{lemma:Fibonacci-sign}, if $n \in \Z$, then $F_n = 0$ if and only if $n = 0$.
\end{proof}

The following lemma describes the ordering of $F_i F_j$'s with $i,j \ge 0$ and $i+j$ fixed.
We need this result to prove that the second half of Theorem \ref{thm:main}.
\begin{lemma} \label{lemma:FiFj-ordered}
Let $n \in \Z_{\ge 0}$, then the set
\[
\{F_i F_j ; \, i,j \in \Z_{\ge 0}, \, i+j = n \}
\]
has cardinality $\lceil \frac{n+1}{2} \rceil$, and if $m \in \{1, \dotsc, \lceil \frac{n+1}{2} \rceil\}$, then the $m$th smallest element of this set is $F_{2m-2-c} F_{n-(2m-2-c)}$, where
\[
c = \begin{cases}
0 & \textrm{ if } m \le \lfloor \frac{n+4}{4} \rfloor \textrm{ or } 2 \nmid n, \\
1 & \textrm{ if } m > \lfloor \frac{n+4}{4} \rfloor \textrm{ and } 2 \mid n.
\end{cases}
\]
\end{lemma}
\begin{proof}
The last part of Lemma \ref{lemma:FiFj-FkFl} implies that the cardinality is $\lceil \frac{n+1}{2} \rceil$.
For $m \in \{1, \dotsc, \lceil \frac{n+1}{2} \rceil\}$ let $c(m)$ be defined as $c$ in the statement above, and let $u(m) = 2m-2-c(m)$ and $v(m) = n-(2m-2-c(m))$.
Then $u(m), v(m) \ge 0$ and $u(m)+v(m) = n$, so $F_{u(m)} F_{v(m)}$ is an element of the set.
Let $1 \le m < m' \le \lceil \frac{n+1}{2} \rceil$.
All we need to prove is that $F_{u(m)} F_{v(m)} < F_{u(m')} F_{v(m')}$.
The difference is
\[
F_{u(m')} F_{v(m')} - F_{u(m)} F_{v(m)} = (-1)^{u(m)} F_{u(m')-u(m)} F_{v(m')-u(m)}
\]
by Lemma \ref{lemma:FiFj-FkFl}.
Here
\[
u(m')-u(m) = 2(m'-m) - c(m') + c(m) \ge 2 - 1 > 0,
\]
so $F_{u(m') - u(m)} > 0$.
Therefore we need to prove that $(-1)^{c(m)} F_{v(m')-u(m)} > 0$.
First suppose that $2 \nmid n$.
Then $c(m) = c(m') = 0$, so $v(m')-u(m) = n - 2(m'+m-2)$ is odd, hence $(-1)^{c(m)} F_{v(m')-u(m)} > 0$ by Lemma \ref{lemma:Fibonacci-sign}.
So let $2 \mid n$.

Suppose that $m > \lfloor \frac{n+4}{4} \rfloor$.
Then $c(m) = c(m') = 1$, and $m > \lfloor \frac{n+4}{4} \rfloor = \lceil \frac{n+2}{4} \rceil$, so $m-1 \ge \lceil \frac{n+2}{4} \rceil \ge \frac{n+2}{4} > \frac{n}{4}$, hence
\[
2 \mid v(m')-u(m) = n - 2(m'+m-3) \le n - 2(2m+1-3) = n - 4(m-1) < 0.
\]
Thus $(-1)^{c(m)} F_{v(m')-u(m)} > 0$ by Lemma \ref{lemma:Fibonacci-sign}.

Finally, let $m \le \lfloor \frac{n+4}{4} \rfloor$.
Then $c(m) = 0$.
If $m' > \lfloor \frac{n+4}{4} \rfloor$, then $c(m') = 1$ and $2 \nmid v(m')-u(m)$, so $F_{v(m')-u(m)} > 0$.
If $m' \le \lfloor \frac{n+4}{4} \rfloor$, then $c(m') = 0$ and $4m'-4 \le n$, so
\[
v(m')-u(m) = n-2(m'+m-2) \ge n - 2(2m'-3) > n - (4m'-4) \ge 0,
\]
hence $F_{v(m')-u(m)}>0$.
\end{proof}

Now we turn to the discussion of alternating binary expansions.
For $d, l_0 \in \Z_{\ge 0}$, $l_1, \dotsc, l_d \in \Z_{\ge 1}$ we define
\[
A(l_0, \dotsc, l_d) = \sum_{i=0}^d (-1)^{d-i} 2^{l_0 + \dotsm + l_i}.
\]
We call this an alternating binary expansion.
The following lemma gives a bound for $A(l_0, \dotsc, l_d)$.
\begin{lemma} \label{lemma:alt-bin-bound}
If $d, l_0 \in \Z_{\ge 0}$ and $l_1, \dotsc, l_d \in \Z_{\ge 1}$, then
\[
2^{l_0+\dotsm+l_d-1} \le A(l_0, \dotsc, l_d) \le 2^{l_0+\dotsm+l_d}.
\]
If $d > 0$, then $A(l_0, \dotsc, l_d) < 2^{l_0+\dotsm+l_d}$.
\end{lemma}
\begin{proof}
We prove by induction on $d$.
For $d = 0$ this is trivial, so suppose that $d \ge 1$ and that the statement is true for smaller values of $d$.
Let $k_i = l_0 + \dotsm + l_i$ for every $i \in \{0, 1, \dotsc, d\}$.
Then $0 \le k_0 < k_1 < \dotsm < k_d$, and $A(l_0, \dotsc, l_d) = 2^{k_d} - 2^{k_{d-1}} + 2^{k_{d-2}} - \dotsm + (-1)^d 2^{k_0} = 2^{k_d} - A(l_0, \dotsc, l_{d-1})$.
Here $0 < 2^{k_{d-1}-1} \le A(l_0, \dotsc, l_{d-1}) \le 2^{k_{d-1}} \le 2^{k_d - 1}$ by the induction hypothesis, so $2^{k_d - 1} \le A(l_0, \dotsc, l_d) < 2^{k_d}$.
\end{proof}

The following lemma describes the number of alternating binary expansions of a positive integer.

\begin{lemma} \label{lemma:alt-bin-uniqueness}
Every $n \in \Z_{\ge 1}$ has exactly two alternating binary expansions, and exactly one of these has the form $n = A(l_0, 1, l_2, \dotsc, l_d)$, where $l_0 \in \Z_{\ge 0}$ and $d, l_2, \dotsc, l_d \in \Z_{\ge 1}$.
If $n$ is a power of $2$, then $d=1$ and the other expansion is $n = A(l_0)$, while otherwise $d \ge 2$ and the other expansion is $A(l_0, l_2+1, l_3, \dotsc, l_d)$.
\end{lemma}
\begin{proof}
It is easy to check that $A(l_0, 1) = A(l_0) = 2^{l_0}$ and $A(l_0, 1, l_2, \dotsc, l_d) = A(l_0, l_2+1, l_3, \dotsc, l_d)$ for $d \ge 2$.
Hence it is enough to prove that every $n \in \Z_{\ge 1}$ has exactly one alternating binary expansion of the form $n = A(l_0, \dotsc, l_d)$ such that $d \ge 1$ and $l_1 = 1$.
We prove by induction on $n$.
Let $k$ be the unique positive integer such that $2^{k-1} \le n < 2^k$.
Let $k_i = l_0 + \dotsm + l_i$ for $i \in \{0, \dotsc, d\}$.
Then $2^{k_d} \le n < 2^{k_d}$ by Lemma \ref{lemma:alt-bin-bound}, so $k_d = k$.
Moreover $n = 2^k - A(l_0, \dotsc, l_{d-1})$, so $n' = A(l_0, \dotsc, l_{d-1}) = 2^k - n$.
If $d = 1$, then $l_0 = k-1$ and $n = A(l_0,1) = 2^{k-1}$.
Now let $d \ge 2$.
Then $A(l_0, \dotsc, l_{d-1}) < 2^{k_{d-1}} \le 2^{k-1}$ by Lemma \ref{lemma:alt-bin-bound}, so $2^{k-1} < n < 2^k$ and $n' < 2^{k-1} < n$.
By the induction hypothesis, there is a unique expansion $n' = A(l_0, \dotsc, l_{d-1})$ with $l_1 = 1$, where furthermore $l_0 + \dotsm + l_{d-1} \le k-1$, because $n' < 2^{k-1}$.
Then $l_d = k - (l_0 + \dotsm + l_{d-1}) \in \Z_{\ge 1}$ is also determined.
\end{proof}

Now we will recall some basic facts about continuants.
Continuants are polynomials in several variables, which come up often when working with continued fractions.
We define the $d$th continuant $K_d(X_1, \dotsc, X_d) \in \Z[X_1, \dotsc, X_d]$ for every $d \in \Z_{\ge 0}$, by the following recursion: $K_0 = 1$, $K_1(X_1) = X_1$, and
\[
K_d(X_1, \dotsc, X_d) = X_d K_{d-1}(X_1, \dotsc, X_{d-1}) + K_{d-2}(X_1, \dotsc, X_{d-2})
\]
for every $d \ge 2$.
We can safely write $K(X_1, \dotsc, X_d)$ instead of $K_d(X_1, \dotsc, X_d)$, because $d$ is anyway determined by the number of variables.
If $\underline{X} = (X_1, \dotsc, X_d)$, then we will also use the notation $K(\underline{X}) = K(X_1, \dotsc, X_d)$.
The continuants are related to continued fractions by the following identity:
\[
[a_0, a_1, \dotsc, a_d] = a_0 + \frac{1}{a_1 + \frac{1}{\ddots + \frac{1}{a_d}}} = \frac{K(a_0, \dotsc, a_d)}{K(a_1, \dotsc, a_d)}
\]
for every $d \ge 0$.

The following lemma gives maybe the most practical description of continuants, using $2 \times 2$ matrices.
\begin{lemma} \label{lemma:continuants-with-matrices}
If $d \ge 2$, then
\[
\begin{pmatrix}
K_d(X_1, \dotsc, X_d) & K_{d-1}(X_1, \dotsc, X_{d-1}) \\
K_{d-1}(X_2, \dotsc, X_d) & K_{d-2}(X_2, \dotsc, X_{d-1})
\end{pmatrix} =
\begin{pmatrix} X_1 & 1 \\ 1 & 0 \end{pmatrix} \dotsm \begin{pmatrix} X_d & 1 \\ 1 & 0 \end{pmatrix}.
\]
So if $d \ge 0$, then $K_d(X_1, \dotsc, X_d) = M_{1,1}$, where $M = \smallmat{X_1}{1}{1}{0} \dotsm \smallmat{X_d}{1}{1}{0}$.
\end{lemma}
\begin{proof}
This easily follows from the defining recursion by induction on $d$.
\end{proof}

The continuants have the following symmetry property.
\begin{lemma} \label{prop:continuant-symmetry}
$K(X_1, \dotsc, X_d) = K(X_d, \dotsc, X_1)$ for every $d \ge 0$.
\end{lemma}
\begin{proof}
Let $M = \smallmat{X_1}{1}{1}{0} \dotsm \smallmat{X_d}{1}{1}{0}$, then
\[
M^{\transpose} = \smallmat{X_d}{1}{1}{0}^{\transpose} \dotsm \smallmat{X_1}{1}{1}{0}^{\transpose} = \smallmat{X_d}{1}{1}{0} \dotsm \smallmat{X_1}{1}{1}{0},
\]
hence
\[
K(X_d, \dotsc, X_1) = (M^{\transpose})_{1,1} = M_{1,1} = K(X_1, \dotsc, X_d)
\]
by Lemma \ref{lemma:continuants-with-matrices}.
\end{proof}

The following lemma states a few simple identities involving continuants.
\begin{lemma} \label{lemma:K-easy-lemma}
If $d \ge 1$, then
\[
K(X_1, \dotsc, X_d) = K(X_1-1, X_2, \dotsc, X_d) + K(X_2, \dotsc, X_d).
\]
If $d \ge 2$, then
\[
K(1, X_2, \dotsc, X_d) = K(X_2, \dotsc, X_d) + K(X_3, \dotsc, X_d).
\]
If $d \ge 1$, then
\[
K(1, X_1, X_2, \dotsc, X_d) = K(X_1+1,X_2, \dotsc, X_d)
\]
and
\[
K(X_1, \dotsc, X_{d-1}, X_d, 1) = K(X_1, \dotsc, X_{d-1}, X_d + 1).
\]
\end{lemma}
\begin{proof}
The identities are trivial for $d = 1$, so assume that $d \ge 2$.
Using Lemma \ref{prop:continuant-symmetry} and the defining recursion of the continuants, we obtain
\[
K(X_1, \dotsc, X_d) = X_1 K(X_2, \dotsc, X_d) + K(X_3, \dotsc, X_d).
\]
Substituting $X_1-1$ into $X_1$, we get
\[
K(X_1-1, X_2, \dotsc, X_d) = (X_1-1) K(X_2, \dotsc, X_d) + K(X_3, \dotsc, X_d).
\]
These two equations immediately imply the first part of the lemma.
Substituting $1$ into $X_1$ in the first equation, we obtain the second part of the lemma.

The fourth identity follows from the third one by Lemma \ref{prop:continuant-symmetry}.
Finally,
\[
K(X_1+1, X_2, \dotsc, X_d) = K(X_1, \dotsc, X_d) + K(X_2, \dotsc, X_d) = K(1, X_1, \dotsc, X_d)
\]
by the first two parts of the lemma.
\end{proof}

%==========================================
\section{Connecting the Stern sequence, alternating binary expansions and continuants} \label{sec:connect}
%==========================================

The following proposition describes the connection between the Stern sequence, alternating binary expansions and continuants.
This result plays a central role in this paper: it allows us to reduce the problem of ordering the elements of the Stern sequence to the problem of ordering continuants.
\begin{proposition} \label{prop:s(A(l0,...,ld))=K(l1,...,ld)}
If $d, l_0 \in \Z_{\ge 0}$ and $l_1, \dotsc, l_d \in \Z_{\ge 1}$, then
\[
s(A(l_0, l_1, \dotsc, l_d)) = K(l_1, \dotsc, l_d).
\]
\end{proposition}
\begin{proof}
We prove by induction on $l_0 + \dotsm + l_d$.
So let $S \in \Z_{\ge 0}$, and suppose the statement is true if $l_0 + \dotsm + l_d < S$.
Now let $l_0 + \dotsm + l_d = S$.
If $l_0 > 0$, then $A(l_0, l_1, \dotsc, l_d) = 2^{l_0} A(0, l_1, \dotsc, l_d)$, so $s(A(l_0, l_1, \dotsc, l_d)) = s(A(0, l_1, \dotsc, l_d)) = K(l_1, \dotsc, l_d)$ by the induction hypothesis.
So assume that $l_0 = 0$.

Let $k_i = l_0 + \dotsm + l_i$ for every $i \in \{0, \dotsc, d\}$, and let $n = A(l_0, l_1, \dotsc, l_d) = \sum_{i=0}^d (-1)^{d-i} 2^{k_i}$.
Note that $k_0 = l_0 = 0$, so $k_1 = l_1$.
Here $n > 0$ by Lemma \ref{lemma:alt-bin-bound}, and $n$ is odd, so $n = 2m+1$ for some $m \in \Z_{\ge 0}$.
If $d = 0$, then $n = 1$, and $s(1) = 1 = K()$.
So let $d \ge 1$.

First suppose that $l_1 \ge 2$.
If $d$ is even, then $m = A(l_1-1, l_2, \dotsc, l_d)$ and $m+1 = A(0, l_1-1, l_2, \dotsc, l_d)$, while if $d$ is odd, then $m+1 = A(l_1-1, l_2, \dotsc, l_d)$ and $m = A(0, l_1-1, l_2, \dotsc, l_d)$.
(Note that $l_1-1 \ge 1$.)
So either way we get
\begin{align*}
s(n) &= s(m) + s(m+1) = s(A(0, l_1-1, l_2, \dotsc, l_d)) + s(A(l_1-1, l_2, \dotsc, l_d)) \\
&= K(l_1-1, l_2, \dotsc, l_d) + K(l_2, \dotsc, l_d) = K(l_1, \dotsc, l_d)
\end{align*}
using the induction hypothesis and Lemma \ref{lemma:K-easy-lemma}.

Now suppose that $l_1 = 1$.
If $d = 1$, then $n = A(0,1) = 1$ and $s(n) = 1 = K(1)$.
So let $d \ge 2$.
If $d$ is odd, then $m = A(l_2, \dotsc, l_d)$ and $m+1 = A(0, l_2, \dotsc, l_d)$, while if $d$ is even, then $m+1 = A(l_2, \dotsc, l_d)$ and $m = A(0, l_2, \dotsc, l_d)$.
So either way we get
\begin{align*}
s(n) &= s(m) + s(m+1) = s(A(0, l_2, \dotsc, l_d)) + s(A(l_2, \dotsc, l_d)) \\
&= K(l_2, \dotsc, l_d) + K(l_3, \dotsc, l_d) = K(1, l_2, \dotsc, l_d) = K(l_1, \dotsc, l_d)
\end{align*}
using the induction hypothesis and Lemma \ref{lemma:K-easy-lemma}.
\end{proof}

For $r \in \Z$ let us define
\[
E_r = \{(l_1, \dotsc, l_d); \, d \in \Z_{\ge 1}, l_1, \dotsc, l_d \in \Z_{\ge 1}, \, l_1 = l_d = 1, \, l_1 + \dotsm + l_d = r+1\}
\]
and
\[
E'_r = \{(l_1, \dotsc, l_d); \, d \in \Z_{\ge 1}, l_1, \dotsc, l_d \in \Z_{\ge 1}, \, l_1 = l_d = 1, \, l_1 + \dotsm + l_d \le r+1\}.
\]
Using Lemmas \ref{lemma:alt-bin-bound}, \ref{lemma:alt-bin-uniqueness} and Proposition \ref{prop:s(A(l0,...,ld))=K(l1,...,ld)}, we obtain the following corollary.
\begin{corollary} \label{cor:diatomic-row=continuants}
If $r \in \Z_{\ge 0}$, then the set of values of the $r$th row of Stern's diatomic array is
\[
\{s(n); \, 2^r \le n \le 2^{r+1} \} = \{K(\underline{l}); \, \underline{l} \in E'_r \}.
\]
\end{corollary}
\begin{proof}
Suppose that $\underline{l} \in E'_r$.
Let $l_0 = r+1 - \sum_{i=1}^r l_i$ and $n = A(l_0, \dotsc, l_d)$, then $2^r \le n \le 2^{r+1}$ and $s(n) = K(\underline{l})$ by Lemma \ref{lemma:alt-bin-bound} and Proposition \ref{prop:s(A(l0,...,ld))=K(l1,...,ld)}.
So the left hand side contains the right hand side.
Conversely, let $n \in \{2^r, \dotsc, 2^{r+1}\}$.
If $n = 2^r$ or $n = 2^{r+1}$, then $s(n) = 1 = K(1)$.
So suppose that $2^r < n < 2^{r+1}$.
According to Lemma \ref{lemma:alt-bin-uniqueness}, $n$ has an alternating binary expansion $n = A(l_0, \dotsc, l_d)$ with $d \ge 1$ and $l_1 = 1$.
Here $l_1 + \dotsm + l_d \le l_0 + \dotsm + l_d = r+1$ by Lemma \ref{lemma:alt-bin-bound}, and $s(n) = K(l_1, \dotsc, l_d)$ by Proposition \ref{prop:s(A(l0,...,ld))=K(l1,...,ld)}.
So if $l_d = 1$, then $\underline{l} = (l_1, \dotsc, l_d) \in E'_r$ and $s(n) = K(\underline{l})$.
If $l_d > 1$, then $\underline{l}' = (l_1, \dotsc, l_d - 1, 1) \in E'_r$ and $s(n) = K(\underline{l}) = K(\underline{l'})$ by Lemma \ref{lemma:K-easy-lemma}.
So the right hand side contains the left hand side.
\end{proof}
This corollary implies that $L_m(r)$ is the $m$th largest distinct value in $\{K(\underline{l}); \, \underline{l} \in E'_r \}$.
So we need to compare the continuants $K(\underline{l})$, where $\underline{l} \in E'_r$.

%==========================================
\section{Comparing continuants} \label{sec:comparing}
%==========================================

The following proposition describes a simple operation on $(l_1, \dotsc, l_d)$ that increases $K(l_1, \dotsc, l_d)$.
This operation is our main tool in comparing continuants.
\begin{proposition} \label{prop:continuant-operation}
Let $d, l_1, \dotsc, l_d \in \Z_{\ge 1}$, $j \in \{1, \dotsc, d\}$, and suppose that $l_j = u+v$ for some $u,v \in \Z_{\ge 1}$.
Then
\[
K(l_1, \dotsc, l_d) \le K(l_1, \dotsc, l_{j-1}, u, v, l_{j+1}, \dotsc, l_d),
\]
where equality holds if and only if $j=1$ and $u=1$, or $j=d$ and $v=1$.
\end{proposition}
\begin{proof}
Lemma \ref{lemma:continuants-with-matrices} implies that
\[
K(l_1, \dotsc, l_d) = (P \smallmat{u+v}{1}{1}{0} Q)_{1,1}
\]
and
\[
K(l_1, \dotsc, l_{j-1}, u, v, l_{j+1}, \dotsc, l_d) = (P \smallmat{u}{1}{1}{0} \smallmat{v}{1}{1}{0} Q)_{1,1},
\]
where $P = \smallmat{l_1}{1}{1}{0} \dotsm \smallmat{l_{j-1}}{1}{1}{0}$ and $Q = \smallmat{l_{j+1}}{1}{1}{0} \dotsm \smallmat{l_d}{1}{1}{0}$.
Using
\[
\smallmat{u}{1}{1}{0} \smallmat{v}{1}{1}{0} - \smallmat{u+v}{1}{1}{0} = \smallmat{(u-1)(v-1)}{u-1}{v-1}{1} = \smallvec{u-1}{1} (v-1,1)
\]
we obtain
\begin{align*}
\Delta &= K(l_1, \dotsc, l_{j-1}, u, v, l_{j+1}, \dotsc, l_d) - K(l_1, \dotsc, l_d) = (P \smallvec{u-1}{1} (v-1,1) Q)_{1,1} \\
&= (P_{1,1} (u-1) + P_{1,2})(Q_{1,1} (v-1) + Q_{2,1}).
\end{align*}
Here $P_{1,1}, P_{1,2}, Q_{1,1}, Q_{2,1}, u-1, v-1 \ge 0$, so $\Delta \ge 0$.
Note that $\det P, \det Q \in \{-1,1\}$, since $\det \smallmat{X}{1}{1}{0} = -1$.
So $(P_{1,1}, P_{1,2}) \neq (0,0)$ and $(Q_{1,1}, Q_{2,1}) \neq (0,0)$.
Hence $\Delta = 0$ if and only if $u = 1$ and $P_{1,2} = 0$, or $v = 1$ and $Q_{2,1} = 0$.
It is easy to see that $P_{1,2} = 0$ if and only if $j = 1$, and similarly, $Q_{2,1} = 0$ if and only if $j = d$.
\end{proof}

We introduce a few notations.
Let
\[
h(l_1, \dotsc, l_d) = \sum_{i=1}^d (l_i - 1) = l_1 + \dotsm + l_d - d.
\]
For $r \in \Z$ and $a \in \Z_{\ge 0}$ let
\[
E_{r, a} = \{(l_1, \dotsc, l_d) \in E_r ; \, h(l_1, \dotsc, l_d) = a\}
\]
and
\[
E'_{r, a} = \{(l_1, \dotsc, l_d) \in E'_r ; \, h(l_1, \dotsc, l_d) = a\}.
\]
Then $E'_r = \bigcup_{a \ge 0} E'_{r,a} = \bigcup_{t \le r} \bigcup_{a \ge 0} E_{t,a}$.

For $s, p_0, p_1, \dotsc, p_s \in \Z_{\ge 0}$ let
\[
w_{p_0, \dotsc, p_s}(X_1, \dotsc, X_s) = (\underbrace{1, \dotsc, 1}_{p_0}, X_1, \underbrace{1, \dotsc, 1}_{p_1}, X_2, \dotsc, \underbrace{1, \dotsc, 1}_{p_{s-1}}, X_s, \underbrace{1, \dotsc, 1}_{p_s})
\]
and
\[
\kappa_{p_0, \dotsc, p_s}(X_1, \dotsc, X_s) = K(w_{p_0, \dotsc, p_s}(X_1, \dotsc, X_s)).
\]
For $s = 0$ we simply write
\[
w_{p_0} = (\underbrace{1, \dotsc, 1}_{p_0}) \textrm{ and } \kappa_{p_0} = K(w_{p_0}).
\]

The idea is that starting from $(l_1, \dotsc, l_d) \in E'_r$, and using the operation of Proposition \ref{prop:continuant-operation} several times, and also possibly increasing elements or adding new elements to $(l_1, \dotsc, l_d)$, we can increase $K(l_1, \dotsc, l_n)$ to $K(w_{r+1})$.
If we stop a bit earlier, we get the largest continuants.
The precise statement is described in the following proposition.
\begin{proposition} \label{prop:continuants-reduction-step}
If $r \in \Z_{\ge 0}$ and $\underline{l} \in E'_r \setminus (E_{r,0} \cup E_{r,1})$, then there is an $\underline{m} \in E_{r-1,0} \cup E_{r,2}$ such that $K(\underline{l}) \le K(\underline{m})$.
\end{proposition}
\begin{proof}
We will use several times that if $u_1, \dotsc, u_d \in \Z_{\ge 1}$, then
\[
K(u_1, \dotsc, u_d) \le K(u_1, \dotsc, u_d, 1),
\]
and if $u'_1, \dotsc, u'_d \in \Z_{\ge 1}$ and $u_i \le u'_i$ for every $i$, then
\[
K(u_1, \dotsc, u_d) \le K(u'_1, \dotsc, u'_d).
\]
We prove by induction on $h(\underline{l})$.
If $h(\underline{l}) = 0$, then $\underline{l} = w_d$ for some $d \le r$, so we can take $\underline{m} = w_r \in E_{r-1,0}$.
If $h(\underline{l}) = 1$, then $\underline{l} = w_{p_0, p_1}(2)$ for some $p_0, p_1 \in \Z_{\ge 1}$ with $p_0+p_1 \le r-2$, so we can take $\underline{m} = w_{p_0,r-2-p_0}(3) \in E_{r,2}$.
Finally, let $h(\underline{l}) \ge 2$, and suppose that the statement is true for smaller values of $h$.
Then there is a $j \in \{2, \dotsc, d-1\}$ such that $l_j \ge 2$.
Let
\[
\underline{l'} = (l_1, \dotsc, l_{j-1}, l_j - 1, 1, l_{j+1}, \dotsc, l_d),
\]
then $\underline{l'} \in E'_r$, and $K(\underline{l}) \le K(\underline{l'})$ by Proposition \ref{prop:continuant-operation}.
Moreover $h(\underline{l'}) = h(\underline{l})-1 \ge 1$, so $\underline{l'} \notin E_{r,0}$.
If $\underline{l'} \notin E_{r,1}$, then by the induction hypothesis there is an $\underline{m} \in E_{r-1,0} \cup E_{r,2}$ such that $K(\underline{l}) \le K(\underline{l'}) \le K(\underline{m})$.
So suppose that $\underline{l'} \in E_{r,1}$.
Then $\underline{l} \in E_{r,2}$, so we can take $\underline{m} = \underline{l}$.
\end{proof}

%==========================================
\section{Fibonacci identities} \label{sec:Fibonacci}
%==========================================

Based on Proposition \ref{prop:continuants-reduction-step}, our next goal is to calculate $K(\underline{l})$ for $\underline{l} \in E_{r,0} \cup E_{r,1} \cup E_{r-1,0} \cup E_{r,2}$.
If $r \ge 1$, then $E_{r,0} = \{w_{r+1}\}$, $E_{r-1,0} = \{w_r\}$,
\[
E_{r,1} = \{w_{p_0,p_1}(2) ; \, p_0, p_1 \in \Z_{\ge 1}, \, p_0 + p_1 = r-1 \},
\]
and $E_{r,2} = U_r \cup V_r$, where
\[
U_r = \{w_{p_0,p_1}(3) ; \, p_0, p_1 \in \Z_{\ge 1}, \, p_0 + p_1 = r-2 \}
\]
and
\[
V_r = \{w_{p_0,p_1,p_2}(2) ; \, p_0, p_2 \in \Z_{\ge 1}, \, p_1 \in \Z_{\ge 0}, \, p_0 + p_1 + p_2 = r-3 \}.
\]
In general $\kappa_{p_0, \dotsc, p_s}(X_1, \dotsc, X_s) = M_{1,1}$, where
\[
M = \smallmat{1}{1}{1}{0}^{p_0} \smallmat{X_1}{1}{1}{0} \smallmat{1}{1}{1}{0}^{p_1} \smallmat{X_2}{1}{1}{0} \dotsm \smallmat{X_s}{1}{1}{0} \smallmat{1}{1}{1}{0}^{p_s}.
\]
Note that $\smallmat{1}{1}{1}{0}^p = \smallmat{F_{p+1}}{F_p}{F_p}{F_{p-1}}$ for every $p \in \Z$.
Calculating the matrix products, we get
\[
\kappa_{p_0} = F_{p_0+1},
\]
\[
\kappa_{p_0,p_1}(2) = 2 F_{p_0+1} F_{p_1+1} + F_{p_0+1} F_{p_1} + F_{p_0} F_{p_1+1},
\]
\[
\kappa_{p_0,p_1}(3) = 3 F_{p_0+1} F_{p_1+1} + F_{p_0+1} F_{p_1} + F_{p_0} F_{p_1+1},
\]
\begin{align*}
\kappa_{p_0,p_1,p_2}(2,2) &= F_{p_0+1} F_{p_1} F_{p_2} + F_{p_0} F_{p_1+1} F_{p_2} + 2 F_{p_0+1} F_{p_1+1} F_{p_2} \\
&\phantom{={}} + F_{p_0+1} F_{p_1-1} F_{p_2+1} + F_{p_0} F_{p_1} F_{p_2+1} + 4 F_{p_0+1} F_{p_1} F_{p_2+1} \\
&\phantom{={}} + 2 F_{p_0} F_{p_1+1} F_{p_2+1} + 4 F_{p_0+1} F_{p_1+1} F_{p_2+1}.
\end{align*}
In the following lemma we express these values in more useful forms.
\begin{lemma} \label{lemma:Fibonacci-identities}
If $p_0, p_1, p_2 \in \Z_{\ge 0}$, then
\[
\kappa_{p_0,p_1}(2) = F_{p_0+p_1+3} - F_{p_0} F_{p_1} = F_{p_0+p_1+2} + F_{p_0+p_1} + F_{p_0-1} F_{p_1-1},
\]
\[
\kappa_{p_0,p_1}(3) = F_{p_0+p_1+3} + F_{p_0+p_1+1} - 2 F_{p_0+p_1-2} - 2 F_{p_0-2} F_{p_1-2},
\]
\begin{align*}
\kappa_{p_0,p_1,p_2}(2,2) &= (F_{p_0+p_1+p_2+4} + F_{p_0+p_1+p_2+2} - F_{p_0+p_1+p_2-4}) - (F_{p_1}(F_{p_0-1} F_{p_2-1} \\
&\phantom{={}} + 3 F_{p_0-2} F_{p_2-1} + 3 F_{p_0-1} F_{p_2-2}) + 2 F_{p_0-2} F_{p_1+1} F_{p_2-2}).
\end{align*}
\end{lemma}
\begin{proof}
If $i,j \in \Z$, then
\[
\smallmat{F_{i+j+1}}{F_{i+j}}{F_{i+j}}{F_{i+j-1}} = \smallmat{1}{1}{1}{0}^{i+j} = \smallmat{1}{1}{1}{0}^i \smallmat{1}{1}{1}{0}^j = \smallmat{F_{i+1}}{F_i}{F_i}{F_{i-1}} \smallmat{F_{j+1}}{F_j}{F_j}{F_{j-1}},
\]
hence $F_{i+j} = F_{i+1} F_j + F_i F_{j-1}$.
Applying this identity twice, we get that $F_{i+j+k} = F_{i+j+1} F_k + F_{i+j} F_{k-1} = (F_{i+1} F_{j+1} + F_i F_j) F_k + (F_{i+1} F_j + F_i F_{j-1}) F_{k-1}$ for $i,j,k \in \Z$.
Using these identities, one can express every term in the statement of the Lemma as a polynomial of $F_{p_0}, F_{p_0+1}, F_{p_1}, F_{p_1+1}, F_{p_2}, F_{p_2+1}$.
Comparing the obtained polynomials, one can check the stated identities.
The calculations could be done by hand, but they are tedious.
Instead we have used Mathematica \cite{Mathematica} to carry out these symbolic calculations, this is done in the attached file.
\end{proof}

\begin{corollary} \label{cor:max-min}
If $r \ge 6$, then
\begin{align*}
\max \{K(\underline{l}) ; \, \underline{l} \in E_{r,1} \} &< F_{r+2} = \kappa_{r+1}, \\
\min \{K(\underline{l}) ; \, \underline{l} \in E_{r,1} \} &= \kappa_{1,r-2}(2) = F_{r+1} + F_{r-1}, \\
\max \{K(\underline{l}) ; \, \underline{l} \in U_r \} &= \kappa_{2,r-4}(3) = F_{r+1} + F_{r-1} - 2F_{r-4}, \\
\max \{K(\underline{l}) ; \, \underline{l} \in V_r \} &= \kappa_{2,0,r-5}(2,2) = F_{r+1} + F_{r-1} - F_{r-7}.
\end{align*}
So if $r \ge 6$, then
\[
\max \{K(\underline{l}) ; \, \underline{l} \in E_{r-1,0} \cup E_{r,2} \} \le \min \{K(\underline{l}) ; \, \underline{l} \in E_{r,1} \},
\]
with strict inequality for $r \neq 7$.
\end{corollary}
\begin{proof}
Let $r \ge 6$.
Recalling the description of $E_{r,1}$, $U_r$ and $V_r$, we see that
\begin{align*}
\{K(\underline{l}) ; \, \underline{l} \in E_{r,1} \} &= \{\kappa_{p_0,p_1}(2) ; \, p_0, p_1 \in \Z_{\ge 1}, \, p_0 + p_1 = r-1\}, \\
\{K(\underline{l}) ; \, \underline{l} \in U_r \} &= \{\kappa_{p_0,p_1}(3) ; \, p_0, p_1 \in \Z_{\ge 1}, \, p_0 + p_1 = r-2 \}, \\
\{K(\underline{l}) ; \, \underline{l} \in V_r \} &= \{\kappa_{p_0,p_1,p_2}(2) ; \, p_0, p_2 \in \Z_{\ge 1}, \, p_1 \in \Z_{\ge 0}, \, p_0 + p_1 + p_2 = r-3 \}.
\end{align*}
Using Lemma \ref{lemma:Fibonacci-identities} one can easily reduce the first part of the proposition to the following statements.
If $p_0, p_1 \in \Z_{\ge 1}$ and $p_0 + p_1 = r-1$, then $F_{p_0} F_{p_1} > 0$ and $F_{p_0-1} F_{p_1-1} \ge 0$.
If $p_0, p_1 \in \Z_{\ge 1}$ and $p_0 + p_1 = r-2$, then $2 F_{p_0-2} F_{p_1-2} \ge 0$.
If $p_0, p_2 \in \Z_{\ge 1}$, $p_1 \in \Z_{\ge 0}$ and $p_0 + p_1 + p_2 = r-3$, then $
F_{p_1}(F_{p_0-1} F_{p_2-1} + 3 F_{p_0-2} F_{p_2-1} + 3 F_{p_0-1} F_{p_2-2}) + 2 F_{p_0-2} F_{p_1+1} F_{p_2-2} \ge 0$.
These statements follow from the facts that $F_n > 0$ for $n \ge 1$, and $F_n \ge 0$ for $n \ge -1$.

Now we prove the last part.
If $\underline{l} \in E_{r-1,0}$, then $K(\underline{l}) = \kappa_r = F_{r+1} < F_{r+1} + F_{r-1}$.
Since $E_{r,2} = U_r \cup V_r$, the statement follows from $F_{r+1} + F_{r-1} - 2F_{r-4} < F_{r+1} + F_{r-1}$ and $F_{r+1} + F_{r-1} - F_{r-7} \le F_{r+1} + F_{r-1}$.
Note that here $F_{r-7} > 0$ if $r \ge 6$ and $r \neq 7$.
\end{proof}

Now we are ready to prove our main result.
\begin{proof}[Proof of Theorem \ref{thm:main}]
The second part of the theorem follows from the first part by Lemma \ref{lemma:FiFj-ordered}.
We prove now the first part.
For $r \in \{0,1,2,3,4,5\}$ one could check the statement by hand.
The attached Mathematica file contains a program that does this.

Suppose that $r \ge 6$, and let $m \ge 1$.
By Corollary \ref{cor:diatomic-row=continuants}, $L_m(r)$ is the $m$th largest distinct value in $\{K(\underline{l}) ; \underline{l} \in E'_r\}$.
If $\underline{l} \in E'_r \setminus (E_{r,0} \cup E_{r,1})$, then
\[
K(\underline{l}) \le \min(K(\underline{l'}) ; \, \underline{l'} \in E_{r,1})
\]
by Proposition \ref{prop:continuants-reduction-step} and Corollary \ref{cor:max-min}.
Moreover
\[
\{K(\underline{l}); \, \underline{l} \in E_{r,0}\} = \{F_{r+2}\}
\]
and
\[
\{K(\underline{l}) ; \, \underline{l} \in E_{r,1}\} = \{F_{r+2} - F_i F_j ; \, i,j \in \Z_{\ge 1}, \, i+j = r-1\}
\]
by Lemma \ref{lemma:Fibonacci-identities}.
So $\{L_1(r), \dotsc, L_{|H|}(r)\} = H$, where
\[
H = \{F_{r+2} - F_i F_j ; \, i,j \in \Z_{\ge 0}, \, i+j = r-1\}.
\]
Here $|H| = \lceil \frac{r}{2} \rceil$ by Lemma \ref{lemma:FiFj-ordered}.
\end{proof}

%==========================================
\section{Further research} \label{sec:further-research}
%==========================================

Using our results, it would not be hard to describe the exact positions where the first $\lceil \frac{r}{2} \rceil$ largest values in the $r$th row of Stern's diatomic sequence appear.

To determine $L_m(r)$ for $1 \le m \le \lceil \frac{r}{2} \rceil$, we needed to calculate and (sometimes) compare the values of $K(\underline{l})$ for $l \in E_{r,0} \cup E_{r,1} \cup E_{r,2}$.
To go a step further, i.e., to determine $L_m(r)$ for some $m > \lceil \frac{r}{2} \rceil$, we would probably need to calculate and compare the values of $K(\underline{l})$ for $l \in \bigcup_{i=0}^3 E_{r,i}$.
For example, we have ordered $\{K(\underline{l}); \, \underline{l} \in E_{r,0} \cup E_{r,1} \}$, but we have not yet ordered $\{K(\underline{l}); \, \underline{l} \in E_{r,2} \}$.

Instead of studying Stern's diatomic array, which starts with the $0$th row $1,1$, we could study the following generalization.
Start with the $0$th row $a,b$, where $a,b \in \R$, and in each step construct a new row by copying the last row, and writing between each two consecutive elements their sum.
One could try to understand the largest values in the $r$th row of this array.


\begin{thebibliography}{9}

\bibitem{Lan14}
J. Lansing, Largest values for the Stern sequence, \textit{J. Integer Seq.}, \textbf{17} (2014), \href{https://cs.uwaterloo.ca/journals/JIS/VOL17/Lansing/lansing2.pdf}{Article 14.7.5}.

\bibitem{Lehmer} D. H. Lehmer, On {S}tern's {D}iatomic {S}eries, \textit{Amer. Math. Monthly}, \textbf{36} (1929), 59--67.

\bibitem{Lucas} E. Lucas, Sur les suites de {F}arey, \textit{Bull. Soc. Math. France}, \textbf{6} (1878), 118--119.

\bibitem{Stern} M. Stern, \"Uber eine zahlentheoretische {F}unktion, \textit{J. Reine Angew. Math.}, \textbf{55} (1858), 193--220.

\bibitem{wikiVajda} \url{https://proofwiki.org/wiki/Vajda's_Identity}.

\bibitem{Mathematica} Wolfram Research, Inc., Mathematica, Version 10.0, Champaign, IL (2014).

\end{thebibliography}
\end{document}